\documentclass[12pt]{article}
\usepackage{amssymb, amscd, amsthm, amsmath,latexsym, amstext,mathrsfs,verbatim,longtable,array,multirow}
\usepackage[all]{xy}
\usepackage{graphicx,enumerate, url}
\usepackage{tikz}

\def\lijntje{\vrule height2.4pt depth-2pt width0.5in}

\def\vlijntje{\vrule height0.45in depth0.4pt width0.4pt}
\def\vlijn{\buildrel {\hbox to 0pt{\hss$\textstyle\circ$\hss}}\over\vlijntje}

\def\dlijntje{{\vrule height2pt depth-1.6pt
width0.5in}\llap{\vrule height4pt depth-3.6pt width0.5in}}
\def\tlijntje{{\vrule height1.7pt depth-1.3pt
width0.5in}\llap{\vrule height3.0pt depth-2.6pt width0.5in}\llap{\vrule height4.3pt depth-3.9pt width0.5in}
}
\def\vtriple#1\over#2\over#3{\mathrel{\mathop{\kern0pt #2}\limits_{\hbox
to 0pt{\hss$#1$\hss}}^{\hbox to 0pt{\hss$#3$\hss}}}}
\def\rvtriple#1\over#2\over#3{\mathrel{\mathop{\kern0pt #2}\limits_{\hbox
to 0pt{\hss$#3$\hss}}^{\hbox to 0pt{\hss$#1$\hss}}}}
\def\Ct{\vtriple{\scriptstyle 2}\over\circ\over{}
\kern-1pt\lijntje\kern-1pt\vtriple{\scriptstyle 1}\over\circ\over{}
\kern-4pt{\dlijntje \kern -25pt<}\kern8pt
\vtriple{\scriptstyle 0}\over\circ\over{}\kern-1pt
}
\def\Bt{\vtriple{\scriptstyle 2}\over\circ\over{}
\kern-1pt\lijntje\kern-1pt\vtriple{\scriptstyle 1}\over\circ\over{}
\kern-4pt{\dlijntje \kern -25pt>}\kern8pt
\vtriple{\scriptstyle 0}\over\circ\over{}\kern-1pt}

\def\ddI{{\rm I}}

\newcommand{\C}{\mathbb C}

\def\Dm{\vtriple{\scriptstyle n+1}\over\circ\over{}\kern-1pt\lijntje\kern-1pt
\vtriple{\scriptstyle{n}}\over\circ\over{}
\cdots\cdots\vtriple{\scriptstyle 4}\over\circ\over{}\kern-1pt\lijntje\kern-1pt
\vtriple{\scriptstyle 3}\over\circ\over{\buildrel
{\scriptstyle 2}\over\vlijn}\kern-1pt\lijntje\kern-1pt
\vtriple{1}\over\circ\over{}\kern-1pt}

\def\Dn{\vtriple{\scriptstyle n}\over\circ\over{}\kern-1pt\lijntje\kern-1pt
\vtriple{\scriptstyle{n-1}}\over\circ\over{}
\cdots\cdots\vtriple{\scriptstyle 4}\over\circ\over{}\kern-1pt\lijntje\kern-1pt
\vtriple{\scriptstyle 3}\over\circ\over{\buildrel
{\scriptstyle 2}\over\vlijn}\kern-1pt\lijntje\kern-1pt
\vtriple{1}\over\circ\over{}\kern-1pt}
\def\En{\vtriple{\scriptstyle n}\over\circ\over{}\kern-1pt\lijntje\kern-1pt
\vtriple{\scriptstyle{n-1}}\over\circ\over{}
\cdots\cdots\vtriple{\scriptstyle 5}\over\circ\over{}\kern-1pt\lijntje\kern-1pt
\vtriple{\scriptstyle 4}\over\circ\over{\buildrel
{\scriptstyle 2}\over\vlijn}\kern-1pt\lijntje\kern-1pt
\vtriple{\scriptstyle 3}\over\circ\over{}\kern-1pt\lijntje\kern-1pt
\vtriple{\scriptstyle 1}\over\circ\over{}\kern-1pt}
\def\An{\vtriple{\scriptstyle n}\over\circ\over{}\kern-1pt\lijntje\kern-1pt
\vtriple{\scriptstyle{n-1}}\over\circ\over{}\kern-1pt\lijntje\kern-1pt
\vtriple{\scriptstyle n-2}\over\circ\over{}
\cdots\cdots
\vtriple{\scriptstyle 2}\over\circ\over{}\kern-1pt\lijntje\kern-1pt
\vtriple{\scriptstyle 1}\over\circ\over{}\kern-1pt}
\def\Cn{\vtriple{\scriptstyle n-1}\over\circ\over{}
\kern-1pt\lijntje\kern-1pt\vtriple{\scriptstyle{n-2}}\over\circ\over{}
\cdots\cdots
\vtriple{\scriptstyle 2}\over\circ\over{}
\kern-1pt\lijntje\kern-1pt\vtriple{\scriptstyle 1}\over\circ\over{}
\kern-4pt{\dlijntje \kern -25pt<}\kern10pt
\vtriple{\scriptstyle 0}\over\circ\over{}\kern-1pt}
\def\Ct{\vtriple{\scriptstyle 2}\over\circ\over{}
\kern-1pt\lijntje\kern-1pt\vtriple{\scriptstyle 1}\over\circ\over{}
\kern-4pt{\dlijntje \kern -25pt<}\kern12pt
\vtriple{\scriptstyle 0}\over\circ\over{}\kern-1pt
}
\def\Bn{\vtriple{\scriptstyle n-1}\over\circ\over{}
\kern-1pt\lijntje\kern-1pt\vtriple{\scriptstyle{n-2}}\over\circ\over{}
\cdots\cdots
\vtriple{\scriptstyle 2}\over\circ\over{}
\kern-1pt\lijntje\kern-1pt\vtriple{\scriptstyle 1}\over\circ\over{}
\kern-4pt{\dlijntje \kern -25pt>}\kern10pt
\vtriple{\scriptstyle 0}\over\circ\over{}\kern-1pt}
\def\Bt{\vtriple{\scriptstyle 2}\over\circ\over{}
\kern-1pt\lijntje\kern-1pt\vtriple{\scriptstyle 1}\over\circ\over{}
\kern-4pt{\dlijntje \kern -25pt>}\kern12pt
\vtriple{\scriptstyle 0}\over\circ\over{}\kern-1pt}
\def\Es{\vtriple{\scriptstyle 6}\over\circ\over{}\kern-1pt\lijntje\kern-1pt
\vtriple{\scriptstyle 5}\over\circ\over{}\kern-1pt\lijntje\kern-1pt
\vtriple{\scriptstyle 4}\over\circ\over{\buildrel
{\scriptstyle 2}\over\vlijn}\kern-1pt\lijntje\kern-1pt
\vtriple{3}\over\circ\over{}\kern-1pt\lijntje\kern-1pt
\vtriple{\scriptstyle 1}\over\circ\over{}\kern-1pt}
\def\Ff{
\vtriple{\scriptstyle 1}\over\circ\over{}
\kern-1pt\lijntje\kern-1pt\vtriple{\scriptstyle 2}\over\circ\over{}
\kern-4pt{\dlijntje \kern -25pt<}\kern10pt
\vtriple{\scriptstyle 3}\over\circ\over{}\kern-1pt\lijntje\kern-1pt
\vtriple{\scriptstyle 4}\over\circ\over{}
\kern-1pt}
\def\Ht{
\vtriple{\scriptstyle 1}\over\circ\over{}
\kern-1pt\overset{5}{\lijntje}\kern-1pt\vtriple{\scriptstyle 2}\over\circ\over{}
\kern-1pt\lijntje\kern-1pt
\vtriple{\scriptstyle 3}\over\circ\over{}\kern-1pt}
\def\Hf{
\vtriple{\scriptstyle 1}\over\circ\over{}
\kern-1pt\overset{5}{\lijntje}\kern-1pt\vtriple{\scriptstyle 2}\over\circ\over{}
\kern-1pt\lijntje\kern-1pt
\vtriple{\scriptstyle 3}\over\circ\over{}\kern-1pt\lijntje\kern-1pt
\vtriple{\scriptstyle 4}\over\circ\over{}
\kern-1pt}
\def\In{
\vtriple{\scriptstyle 0}\over\circ\over{}
\kern-1pt\overset{n}{\lijntje}\kern-1pt\vtriple{\scriptstyle 1}\over\circ\over{}
\kern-1pt}
\def\Gt{
\vtriple{\scriptstyle 0}\over\circ\over{}
\kern-4pt{\tlijntje\kern -25pt<}\kern 10pt\vtriple{\scriptstyle 1}\over\circ\over{}
\kern-1pt}
\def\EBn{\vtriple{\scriptstyle n-1}\over\circ\over{}
\kern-1pt\lijntje\kern-1pt\vtriple{\scriptstyle{n-2}}\over\circ\over{\buildrel
{\scriptstyle -1}\over\vlijn}\cdots\cdots
\vtriple{\scriptstyle 2}\over\circ\over{}
\kern-1pt\lijntje\kern-1pt\vtriple{\scriptstyle 1}\over\circ\over{}
\kern-4pt{\dlijntje \kern -25pt<}\kern8pt
\vtriple{\scriptstyle 0}\over\circ\over{}\kern-1pt}
\def\Cn{\vtriple{\scriptstyle n-1}\over\circ\over{}
\kern-1pt\lijntje\kern-1pt\vtriple{\scriptstyle{n-2}}\over\circ\over{}
\cdots\cdots
\vtriple{\scriptstyle 2}\over\circ\over{}
\kern-1pt\lijntje\kern-1pt\vtriple{\scriptstyle 1}\over\circ\over{}
\kern-4pt{\dlijntje \kern -25pt<}\kern10pt
\vtriple{\scriptstyle 0}\over\circ\over{}\kern-1pt}
\def\ECn{\vtriple{\scriptstyle -2}\over\circ\over{}
\kern-4pt{\dlijntje \kern -25pt>}\kern8pt\vtriple{\scriptstyle n-1}\over\circ\over{}
\kern-1pt\lijntje\kern-1pt\vtriple{\scriptstyle{n-2}}\over\circ\over{}
\cdots\cdots
\vtriple{\scriptstyle 2}\over\circ\over{}
\kern-1pt\lijntje\kern-1pt\vtriple{\scriptstyle 1}\over\circ\over{}
\kern-4pt{\dlijntje \kern -25pt<}\kern12pt
\vtriple{\scriptstyle 0}\over\circ\over{}\kern-1pt}
\def\Fo{\vtriple{\scriptstyle -1}\over\circ\over{}
\kern-1pt\lijntje\kern-1pt
\vtriple{\scriptstyle 1}\over\circ\over{}
\kern-1pt\lijntje\kern-1pt\vtriple{\scriptstyle 2}\over\circ\over{}
\kern-4pt{\dlijntje \kern -25pt<}\kern8pt
\vtriple{\scriptstyle 3}\over\circ\over{}\kern-1pt\lijntje\kern-1pt
\vtriple{\scriptstyle 4}\over\circ\over{}
\kern-1pt}
\def\Ft{
\vtriple{\scriptstyle 1}\over\circ\over{}
\kern-1pt\lijntje\kern-1pt\vtriple{\scriptstyle 2}\over\circ\over{}
\kern-4pt{\dlijntje \kern -25pt<}\kern8pt
\vtriple{\scriptstyle 3}\over\circ\over{}\kern-1pt\lijntje\kern-1pt
\vtriple{\scriptstyle 4}\over\circ\over{}
\kern-1pt\lijntje\kern-1pt
\vtriple{\scriptstyle -2}\over\circ\over{}
\kern-1pt}
\def\Go{\vtriple{\scriptstyle -1}\over\circ\over{}
\kern-1pt\lijntje\kern-1pt
\vtriple{\scriptstyle 0}\over\circ\over{}
\kern-4pt{\tlijntje\kern -25pt<}\kern 12pt\vtriple{\scriptstyle 1}\over\circ\over{}
\kern-1pt}
\def\Gf{
\vtriple{\scriptstyle 0}\over\circ\over{}
\kern-4pt{\tlijntje\kern -25pt<}\kern 12pt\vtriple{\scriptstyle 1}\over\circ\over{}
\kern-1pt\lijntje\kern-1pt
\vtriple{\scriptstyle -2}\over\circ\over{}
\kern-1pt}

\numberwithin{equation}{section}

\newtheorem{lemma}{Lemma}[section]
\newtheorem{cor}[lemma]{Corollary}

\newtheorem{thm}[lemma]{Theorem}

\theoremstyle{remark}
\newtheorem{rem}[lemma]{Remark}

\theoremstyle{definition}

\newtheorem{defn}[lemma]{Definition}

\begin{document}
\title{Coxeter groups and the proper joint spectrums of their faithful representations}
\author{Shoumin Liu\footnote{The author is  funded by the NSFC (Grant No. 11971181, Grant No.11871308) },   Zhaohuan Peng, Xumin Wang}
\date{}
\maketitle
%\mainmatter
%\setcounter{section}{-1} \tableofcontents

\begin{abstract}
In this paper, we analyze the faithful representations of the dihedral groups, and prove that the Coxeter groups can be determined by the proper joint spectrum of  their faithful representations.
\end{abstract}
Keywords: proper joint spectrum;  Coxeter groups;  dihedral groups;\\
% group representation.\\
MSC:17B05,17B10\\

\section{Introduction}
\hspace{1.3em}The study of Coxeter groups is a  very classical topic in Lie theory and representation theory, which is related to many subjects in mathematics.
The notion of projective spectrum of finite operators was first defined by Yang in \cite{Y2009}, which has played a powerful role in the study of functional analysis,
group representation theory, Lie algebras, and spectral dynamical systems.
 A lot of  research work about them has been done in \cite{CMK2016}, \cite{GY2017},\cite{GR2014}, \cite{HY2018} and \cite{KY2021}. The proper joint spectrum is a special case of the
 projective spectrum, which can build a bridge between operator theory and geometry. There are some results about the Coxeter groups and the  proper joint spectrums of their generators in
 \cite{CST2021} and \cite{S2022}. In \cite[Theorem 1.1]{CST2021}, the authors prove that a Coxeter group  $W$ can be determined through  the joint spectrum associated to the left regular representation of the group  $W$.
 In the proof of the theorem, the author mainly use  geometric and analytic  tools. Here we want to give a proof of a similar conclusion for the Coxeter groups whose Dynkin diagrams just have finite bonds,
 by the proper joint spectrums associated to any faithful representations of the
 Coxeter groups, and we do it in a pure algebraic approach by analyzing the structure of the faithful representation.
% Specifically, proper joint spectrum is an extension of the projective joint spectrum. We first consider the special Coxeter group, the dihedral group. If $\rho$ is its finite dimensional faithful representation, we analysis what conditions need to be satisfied for the irreducible decomposition of $\rho$, and then we discuss the deterministic relation between the faithful representation of Coxeter group $W$ and its proper joint spectrum.

The structure of the paper is as follows.
In Section \ref{sect:BN}, we recall some necessary conceptions for the paper.
In Section \ref{sect:PJSI2n}, we calculate the characteristic polynomial and proper joint spectrum of the irreducible representations of  dihedral groups and summarize the results in $3$ tables;
In Section \ref{sect:FFI2n}, we present the equivalent condition for  a representation $\rho$ of $W(\ddI_{2}(n))$  being a faithful representation through the decomposition of $\rho$ into irreducible representations.
% it is showed that, $\rho$ is a finite dimensional representation of $W(I_{2}(n))$, $\rho=\oplus(\rho_{i})^{t_{i}}$, where $t_{i}$ is the multiplicity of $\rho_{i}$. To make $\rho$ a faithful representation, what conditions should the decomposition of $\rho$ meet;
In Section \ref{sect:MT}, we prove our main theorem, a Coxeter group with finite bonds can be determined by  the proper joint spectrum of an arbitrary faithful representation.

\section{Some basic notions}\label{sect:BN}
We first recall the definition of Coxeter groups.
\begin{defn}\label{Cox}
Let $M=(m_{ij})_{1\leq i,j\leq n}$ be a symmetric $n\times n$ matrix with entries from $\mathbb{N}\cup \infty$ such that $m_{ii}=1$ for all $i\in [n]$ and $m_{ij}>1$ whenever $i\neq j$. The Coxeter group of type $M$ is the group
$$W(M)=<s_{1},...,s_{n}|(s_{i}s_{j})^{m_{ij}}=1, i,j\in [n], m_{ij}<\infty>.$$
We often write $S$ instead of $s_{1},...,s_{n}$ and if no confusion is imminent, $W$ instead of $W(M)$. The pair $(W,S)$ is called the Coxeter system of type $M$.
\end{defn}
In this paper, we just consider the Coxeter group with the bond $m_{ij}$ being finite.\\
We also recall some conceptions from \cite{Y2009}.
\begin{defn}
Suppose $A_{1},...,A_{n}$ are bounded linear operators on  a Hilbert space $V$. The $\mathbf{projective}$ $\mathbf{joint}$ $\mathbf{spectrum}$  of $A_{1},...,A_{n}$ is the set
\begin{eqnarray*}
&&\sigma(A_{1},...,A_{n})\\
&=&\{[x_{1}:...:x_{n}]\in \mathbb{CP}^{n}:x_{1}A_{1}+...+x_{n}A_{n}\, \text{is not invertible}\}.
\end{eqnarray*}
The $\mathbf{proper}$ $\mathbf{joint}$ $\mathbf{spectrum}$ of $A_{1},...,A_{n}$ is the set
\begin{eqnarray*}
&&\sigma_p(A_{1},...,A_{n})\\
&=&\{[x_{1},...,x_{n}]\in \mathbb{C}^{n}:x_{1}A_{1}+...+x_{n}A_{n}-I\, \text{is not invertible}\}.
\end{eqnarray*}
\end{defn}

%Let $T=\{g_{0},...,g_{n}\}$ be a set of generators of a group $W$, and let
%$$\rho:W\longrightarrow GL(V)$$
%be a homomorphism into the group of bounded invertible linear operators on a Hilbert space $V$, $\sigma(\rho(g_{0}),...,\rho(g_{n}))$ is the corresponding joint spectrum.
%
%
%Now let $A_{1},...,A_{n}$ be bounded linear operators on a Hilbert space $V$, we will be considering the projective joint spectrum $\sigma(-I,A_{1},...,A_{n})$ of the tuple $-I,A_{1},...,A_{n}$, where $I$ is the identity operator, and its intersection with the chart $\{[x_{0}:...:x_{n}]|x_{0}\neq 0\}$. By taking $x_{0}=1$ we identity this intersection with a closed subset of $\mathbb{C^{n}}$ called the proper joint spectrum of the tuple $A_{1},...,A_{n}$, denoted by $\sigma_{p}(A_{1},...,A_{n})$; And when $V$ is finite dimensional its defining polynomial is
%$$F(x_{1},...,x_{n})=det(-I+x_{1}A_{1}+...+x_{n}A_{n}).$$

Let $T=\{s_{1},...s_{n}\}$ be a set of generators of the Coxeter group $W$ associated to the Coxeter diagram of $W$, and let
$$\rho:W\longrightarrow GL(V)$$
be a representation of $W$, with $V$ being a complex linear space of finite dimension.
Then
\begin{eqnarray*}
&& \sigma_{p}(\rho(s_{1}),...,\rho(s_{n}))=\\
&&\{(x_{1},...,x_{n})\in \mathbb{C}^{n}|-I+x_{1}\rho(s_{1})+...+x_{n}\rho(s_{n})\text{ is not  invertible}\}
\end{eqnarray*}
is called  the proper joint spectrum of $(W,\rho)$.

\section{Proper joint spectrum of an irreducible representation of $W(\ddI_{2}(n))$}\label{sect:PJSI2n}

\hspace{1.3em}From \cite[Example 8.2.3]{S2011}, for a finite dihedral group, its non-linear irreducible representations  have been clearly described,  and we can easily find their linear representations by
 its generators and their defining relations through Definition \ref{Cox}.
 In this section, we will focus on calculating the characteristic polynomial and proper joint spectrum corresponding to the irreducible representation of the dihedral group,
and we summarize our results in tables, which can be used in the later sections.

Let $W(\ddI_{2}(n))$ represent the dihedral group of order $2n$. For Definition \ref{Cox}, we set
$$\ddI_{2}(n)=
\begin{pmatrix}
1 & n\\
n & 1
\end{pmatrix},$$
$$Dih_{2n}=W(\ddI_{2}(n))=\{ s_{1},s_{2}|s_{1}^{2}=1,s_{2}^{2}=1,(s_{1}s_{2})^{n}=1\}.$$
Suppose $\rho$ is an irreducible representation of $W(\ddI_{2}(n))$. First, we compute the proper joint spectrum of $(W,\rho)$ defined by  $\mathrm{det}(-I+x_{1}\rho(s_{1})+x_{2}\rho(s_{2}))$.

It is known that, for the irreducible representations of $W(\ddI_{2}(n))$, the degree of $\rho$ is 1 or 2. Suppose $W(\ddI_{2}(n))$ has $m_{1}$ irreducible representations of degree one, and $m_{2}$ irreducible representations of degree two. We have
$$|W(\ddI_{2}(n))|=2n=m_{1}+4m_{2}.$$
We use the symbol $\rho_{i,j}^{n}$ to represent them with $i=1$ or $2$ for the dimension of the representation,
 and  $j=1,...,m_{1}$ when $i=1$, $j=1,...,m_{2}$ when $i=2$.

We need to discuss the parity of $n$.

When $n=2$ namely $s_{1}s_{2}=s_{2}s_{1}$. $W(\ddI_{2}(2))=\mathbb{Z}/(2\mathbb{Z})\times \mathbb{Z}/(2\mathbb{Z})$. $W(\ddI_{2}(2))$ has 4 irreducible representations of dimension $1$ in the following,
and we show them in the picture below.
\begin{enumerate}[(i)]
\item $\rho_{1,1}^{2}(s_{1})=\rho_{1,1}^{2}(s_{2})=1$;

\item $\rho_{1,2}^{2}(s_{1})=\rho_{1,2}^{2}(s_{2})=-1$;

\item $\rho_{1,3}^{2}(s_{1})=1, \rho_{1,3}^{2}(s_{2})=-1$;

\item $\rho_{1,4}^{2}(s_{1})=-1, \rho_{1,4}^{2}(s_{2})=1$.
\end{enumerate}
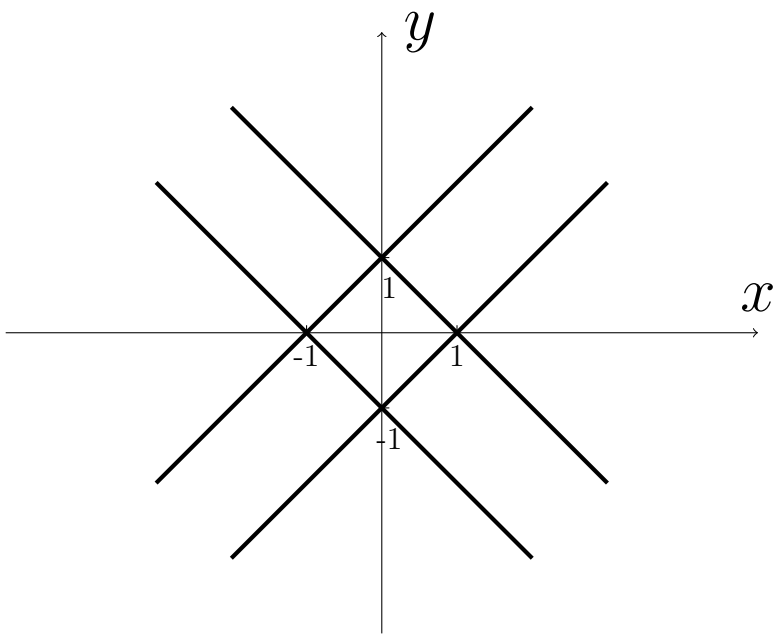
\begin{tikzpicture}
  \draw [->] (-5,0)--(5,0) node[above,scale=2] {$x$};
  \draw [->] (0,-4)--(0,4) node[right,scale=2] {$y$};
  \draw (-1,0)--(-1,0.1) node[below=3.6pt]{-1};
  \draw (1,0)--(1,0.1) node[below=3.6pt]{1};
  \draw (0,1)--(0.1,1) node[below=3.6pt]{1};
  \draw (0,-1)--(0.1,-1) node[below=3.6pt]{-1};
  \draw[ultra thick](-2,3)--(3,-2);
  \draw[ultra thick](-3,-2)--(2,3);
  \draw[ultra thick](-3,2)--(2,-3);
  \draw[ultra thick](-2,-3)--(3,2);
 \end{tikzpicture}

When $2\nmid n$, there are two representations of dimension $1$ as the  following, and $m_{1}=2$.

\begin{enumerate}[(i)]
\item $\rho_{1,1}^{n}(s_{1})=\rho_{1,1}^{n}(s_{2})=1$;
\item $\rho_{1,2}^{n}(s_{1})=\rho_{1,2}^{n}(s_{2})=-1$.
\end{enumerate}

When $2|n$, there are four representations of dimension 1, and $m_{1}=4$.
\begin{enumerate}[(i)]
\item $\rho_{1,1}^{n}(s_{1})=\rho_{1,1}^{n}(s_{2})=1$;
\item $\rho_{1,2}^{n}(s_{1})=\rho_{1,2}^{n}(s_{2})=-1$;
\item $\rho_{1,3}^{n}(s_{1})=1, \rho_{1,3}^{n}(s_{2})=-1$;
\item $\rho_{1,4}^{n}(s_{1})=-1, \rho_{1,4}^{n}(s_{2})=1$.
\end{enumerate}

When $n$ is an odd number,  we have
$$m_{1}=2,\quad m_{2}=\frac{2n-2}{4}=\frac{n-1}{2}.$$
%Make $\rho$ is $\rho_{2,k}^{n}, k=1,2,...,\frac{n-1}{2}$.

When $n$ is an even number, $$m_{1}=4,\quad m_{2}=\frac{2n-4}{4}=\frac{n-2}{2}.$$
%Make $\rho$ is $\rho_{2,k}^{n}, k=1,2,...,\frac{n-2}{2}$.

Next, let us calculate $D_{i,j}=\mathrm{det}[-I+x_{1}\rho_{i,j}^{n}(s_{1})+x_{2}\rho_{i,j}^{n}(s_{2})]$.

When $n$ is an odd number, for the  dimensional representations of dimension $1$, we have
\begin{center}
$D_{1,1}=-1+x_{1}+x_{2}$,
\end{center}
\begin{center}
$D_{1,2}=-1-x_{1}-x_{2}$;
\end{center}

and for the  irreducible representations of dimension $2$,  it follows that
$$\rho_{2,k}^{n}(s_{1})=
\begin{pmatrix}
0 & 1\\
1 & 0
\end{pmatrix}, \quad \rho_{2,k}^{n}(s_{2})=
\begin{pmatrix}
0 & e^{\frac{2\pi ikm}{n}}\\
e^{\frac{-2\pi ikm}{n}} & 0
\end{pmatrix},$$

then
\begin{eqnarray*}
D_{i,j}&=&\mathrm{det}[-I+x_{1}\rho_{i,j}^{n}(s_{1})+x_{2}\rho_{i,j}^{n}(s_{2})]\\
&=&
\begin{vmatrix}
-1 & x_{1}+x_{2}e^{\frac{2\pi ik}{n}}\\
x_{1}+x_{2}e^{\frac{-2\pi ik}{n}} & -1
\end{vmatrix}\\
&=&1-x_{1}^{2}-x_{2}^{2}+2cos\frac{2\pi k}{n} x_{1}x_{2},
\end{eqnarray*}
where $1\leq k\leq \frac{n-1}{2}$.

Similarly, we can deal with the case when $n$ is an even number, and it follows that

\begin{center}
$D_{1,1}=-1+x_{1}+x_{2}$;
\end{center}

\begin{center}
$D_{1,2}=-1-x_{1}-x_{2}$;
\end{center}

\begin{center}
$D_{1,3}=-1+x_{1}-x_{2}$;
\end{center}

\begin{center}
$D_{1,4}=-1-x_{1}+x_{2}$;
\end{center}

\begin{center}
$D_{2,k}=1-x_{1}^{2}-x_{2}^{2}+2cos\frac{2\pi k}{n} x_{1}x_{2}, 1\leq k\leq \frac{n-2}{2}$.
\end{center}
For the general representation $\rho$ for $W(\ddI_{2}(n))$, we use $F_{\rho}^{W(\ddI_{2}(n))}(x_{1},x_{2})$ to represent the equation defining the  proper joint spectrum,
which means that the $\sigma_p(s_1,s_2)$ is defined by
 $$F_{\rho}^{W(\ddI_{2}(n))}(x_{1},x_{2})=det(-I+x_{1}\rho(s_{1})+x_{2}\rho(s_{2}))=0.$$ 
 %and the kernel and image of $\rho$ for $\rho\in Irr(W(\ddI_{2}(n))).$
For our aim in the next section,we also compute the kernel and image for each irreducible representation of $W(\ddI_{2}(n))$, which will be presented in the lemma below.

Now we summarize the above results in the  below.
\begin{lemma} \label{tables}
For the finite dihedral group $W(\ddI_{2}(n))$, its proper joint spectrums for irreducible representations can be presented in the following three tables.
Table $1$ is for the case $n=2$; the  Table $2$ is for the case $n>2$, $2\nmid n$, and $1\leq k\leq \frac{n-1}{2}$; the  Table $3$ is for the case  $n>2$, $2|n$, and $ 1\leq k\leq \frac{n-2}{2}$.

\begin{table}[!ht]
\begin{center}
\begin{tabular}{|c|c|c|c|} \hline %
$\rho$ & $kernel$ & $image$ & $F_{\rho}^{W(\ddI_{2}(n))}=0$\\ \hline
$\rho_{1,1}^{2}$ & $W(\ddI_{2}(n))$         & $\left<1\right>$      & $x_{1}+x_{2}-1=0$\\ \hline
$\rho_{1,2}^{2}$ & $\mathbb{Z}/(2\mathbb{Z})=\left<s_{1}s_{2}\right>$ & $\mathbb{Z}/(2\mathbb{Z})$   & $-x_{1}-x_{2}-1=0$\\ \hline
$\rho_{1,3}^{2}$ & $\mathbb{Z}/(2\mathbb{Z})=\left<s_{2}\right>$      & $\mathbb{Z}/(2\mathbb{Z})$   & $-x_{1}+x_{2}-1=0$\\ \hline
$\rho_{1,4}^{2}$ & $\mathbb{Z}/(2\mathbb{Z})=\left<s_{1}\right>$      & $\mathbb{Z}/(2\mathbb{Z})$   & $x_{1}-x_{2}-1=0$\\ \hline
\end{tabular}
\caption{the case for $W(\ddI_{2}(2))$}
\end{center}
\end{table}

\begin{table}[!ht]
\begin{center}
\begin{tabular}{|c|c|c|c|} \hline %
$\rho$ & $kernel$ & $image$ & $F_{\rho}^{W(\ddI_{2}(n))}=0$\\ \hline
$\rho_{1,1}^{n}$ & $W(\ddI_{2}(n))$         & $\left<1\right>$      & $x_{1}+x_{2}-1=0$\\ \hline
$\rho_{1,2}^{n}$ & $\mathbb{Z}/(2\mathbb{Z})=\left<s_{1}s_{2}\right>$ & $\mathbb{Z}/(2\mathbb{Z})$   & $-x_{1}-x_{2}-1=0$\\ \hline
$\rho_{2,k}^{n}$ & $\mathbb{Z}/((n,k)\mathbb{Z})=\left<(s_{1}s_{2})^{\frac{n}{(n,k)}}\right>$      & $W(\ddI_{2}(\frac{n}{(n.k)}))$   & $x_{1}^{2}+x_{2}^{2}+2cos\frac{2\pi k}{n} x_{1}x_{2}-1=0$\\ \hline
\end{tabular}
\caption{cases for $W(\ddI_{2}(n))$, $n$ being odd}
\end{center}
\end{table}

\begin{table}[!ht]
\begin{center}
\begin{tabular}{|c|c|c|c|} \hline %
$\rho$ & $kernel$ & $image$ & $F_{\rho}^{W(\ddI_{2}(n))}=0$\\ \hline
$\rho_{1,1}^{n}$ & $W(\ddI_{2}(n))$         & $\left<1\right>$      & $x_{1}+x_{2}-1=0$\\ \hline
$\rho_{1,2}^{n}$ & $\mathbb{Z}/(n\mathbb{Z})=\left<s_{1}s_{2}\right>$ & $\mathbb{Z}/(2\mathbb{Z})$   & $-x_{1}-x_{2}-1=0$\\ \hline
$\rho_{1,3}^{n}$ & $W(\ddI_{2}(\frac{n}{2}))=\left<s_{2},s_{1}s_{2}s_{1}\right>$      & $\mathbb{Z}/(2\mathbb{Z})$   & $-x_{1}+x_{2}-1=0$\\ \hline
$\rho_{1,4}^{2}$ & $W(\ddI_{2}(\frac{n}{2}))=\left<s_{1},s_{2}s_{1}s_{2}\right>$      & $\mathbb{Z}/(2\mathbb{Z})$   & $x_{1}-x_{2}-1=0$\\ \hline
$\rho_{2,k}^{n}$ & $\mathbb{Z}/((n,k)\mathbb{Z})=\left<(s_{1}s_{2})^{\frac{n}{(n,k)}}\right>$      & $W(\ddI_{2}(\frac{n}{(n.k)}))$   & $x_{1}^{2}+x_{2}^{2}+2cos\frac{2\pi k}{n} x_{1}x_{2}-1=0$\\ \hline
\end{tabular}
\caption{cases for $W(\ddI_{2}(n))$, $n$ being even}
\end{center}
\end{table}

\end{lemma}

\section{The faithfulness for a  representation $\rho$ of $W(\ddI_{2}(n))$ }\label{sect:FFI2n}
\hspace{1.3em}Suppose $\rho$ is a finite dimensional representation of $W(\ddI_{2}(n))$, and we will  show  a criterion to determine whether the representation is a faithful
 representation by analyzing the decomposition of this representation.

 Suppose $Irr(W(\ddI_{2}(n)))$ denotes all irreducible representations of $W(\ddI_{2}(n))$. Let $\rho=\bigoplus (\rho_{i})^{t_{i}}$, where $t_{i}$ is the multiplicity of $\rho_{i}\in Irr(W(\ddI_{2}(n)))$.
 %Now we decide when $\rho$ is a faithful representation of $W(\ddI_{2}(n))$.
Let $V_{\rho}^{W(\ddI_{2}(n))}$ be the proper joint spectrum defined by $F_{\rho}^{W(\ddI_{2}(n))}=det(-I+x_{1}\rho(s_{1})+x_{2}\rho(s_{2}))=0$.
The following theorem can hold.

\begin{thm}\label{decomp}
For the representations of  $W(\ddI_{2}(n))$, the following holds.
\begin{enumerate}[(i)]
\item For each irreducible representation $\rho$ of $W(\ddI_{2}(n))$, the set $V_{\rho}^{W(I_{2}(n))}$ is a line or an ellipse,
and $V_{\rho_{_{1}}}^{W(\ddI_{2}(n))}\neq V_{\rho_{2}}^{W(\ddI_{2}(n))}$if $\rho_{1}\neq \rho_{2}\in Irr(W(\ddI_{2}(n)))$.
\item For any finite representation $\rho$ of $W(\ddI_{2}(n))$,
the irreducible component of $V_{\rho}^{W(\ddI_{2}(n))}$ is one to one corresponding to the irreducible representation of $W(\ddI_{2}(n))$ occurring in the decomposition of $\rho$.
\end{enumerate}
\end{thm}
\begin{proof}
The conclusion in (i) can be verified by the tables in Lemma \ref{tables}.

For (ii),  we have  $\rho=\bigoplus \rho_{i}^{t_{_{i}}}$ being its decomposition, $\rho_{i}\in Irr(W(\ddI_{2}(n)))$,
which implies the equation $F_{\rho}^{W(\ddI_{2}(n))}$ determining $V_{\rho}^{W(\ddI_{2}(n))}$ having the decomposition

$$F_{\rho}^{W(\ddI_{2}(n))}=\prod (F_{\rho_{i}}^{W(\ddI_{2}(n))})^{t_{i}}.$$

Therefore  $V_{\rho}^{W(\ddI_{2}(n))}$ has irreducible components $V_{\rho_{1}}^{W(\ddI_{2}(n))},...,V_{\rho_{k}}^{W(\ddI_{2}(n))}$ being  a line or an  ellipse, having nothing to do with
the multiplicities $t_i$s.
\end{proof}

Then the corollary holds for the Theorem \ref{decomp}.

\begin{cor}
Let $\rho,\rho^{'}$ be the finite dimensional  representations of $W(\ddI_{2}(n))$. Then $V_{_\rho}^{W(\ddI_{2}(n))}=V_{\rho^{'}}^{W(\ddI_{2}(n))}$ if and only if
the irreducible representations in $W(\ddI_{2}(n))$ occurring in the decomposition of $\rho$ and $\rho^{'}$ are the same.
\end{cor}

For the faithfulness of  $\rho$ for $W(\ddI_{2}(n))$,  we have the following theorem.

\begin{thm}\label{faithful}
The representation $\rho$ of $W(\ddI_{2}(n))$ is a faithful representation  if and only if
the following conditions hold for different $n$.
\begin{enumerate}[(i)]

\item   The representation  $\rho$ has at least 2 of $\rho_{1,2}^{2},\rho_{1,3}^{2},\rho_{1,4}^{2}$ in its irreducible decomposition when $n=2$.

\item  The representation $\rho$ has distinct $\rho_{2,k_{i}}^{n},i=1,...,t$ in its irreducible decomposition with
$((n,k_{1}),...,(n,k_{t}))=1$  when $2\nmid n$.

\item

 The representation $\rho$ has neither $\rho_{1,3}^{n}$ or $\rho_{1,4}^{n}$ in its decomposition, and $\rho$ has $\rho_{2,k_{1}}^{n},...,\rho_{2,k_{t}}^{n}$ with $((n,k_{1}),...,(n,k_{t}))=1$ or
$\rho$ has either $\rho_{1,3}^{n}$ or $\rho_{1,4}^{n}$ in its decomposition, and $\rho$ has $\rho_{2,k_{1}}^{n},...,\rho_{2,k_{t}}^{n}$ with $(2,(n,k_{1}),...,(n,k_{t}))=1$ when $2|n$ .

\end{enumerate}
\end{thm}

\begin{proof}
Suppose $\rho=\rho_{1}^{t_{1}}\oplus \rho_{2}^{t_{2}}\oplus,...\oplus\rho_{k}^{t_{k}}, \rho_{i}\in Irr(W(\ddI_{2}(n)))$,
then $\rho$ is faithful if and only if $\mathrm{ker}(\rho_{1})\cap \mathrm{ker}(\rho_{2})\cap ...\cap \mathrm{ker}(\rho_{k})={1}$.

The case (i) can be verified from Table $1$ in Lemma \ref{tables}.

 Since $W(\ddI_{2}(n))=\left< s_{1},s_{2}|s_{1}^{2}=1,s_{2}^{2}=1,(s_{1}s_{2})^{n}=1\right>$,  write  $r=s_{1}s_{2}$.  Hence the order of $r$ is $n$.

 Let us prove (ii). By Table $2$  in Lemma \ref{tables}, we see that $\mathrm{ker} \rho_{2,k}^{n}\subseteq \mathrm{ker}\rho_{1,2}^{n}\subseteq \mathrm{ker} \rho_{1,1}^{n}$,
and $\mathrm{ker} \rho_{1,2}^{n}\neq \{1\}$. Hence, when $\rho$ is faithful,
$\rho$ must have some $\rho_{2,k}^{n}$ in its decomposition. Suppose for those $2$ dimensional representations,
$\rho$ has $\rho_{2,k_{i}}^{n},i=1,...,t$ in its decomposition. Hence $\mathrm{ker} \rho_{2,k_{i}}^{n}=<r^{\frac{n}{(n,k_{i})}}>$ by Table $2$  in Lemma \ref{tables}. Therefore, it follows that
$$\mathrm{ker} \rho=\mathrm{ker} \rho_{2,k_{1}}^{n}\cap \mathrm{ker}\rho_{2,k_{2}}^{n}\cap...\cap \mathrm{ker} \rho_{2,k_{t}}^{n}=\bigcap\left<r^{\frac{n}{(n,k_{i})}}\right>.$$
Since $\left<r^\frac{n}{(n,k_{i})}\right>$ is a cyclic subgroup in $\left<r\right>$ of order $(n,k_{i})$, when $\rho$ is faithful, it is equivalent to $((n,k_{1}),...,(n,k_{t}))=1$ or $(k_{1},...,k_{t},n)=1$.

Now we prove (iii). For the case $1$ of (iii), the argument is similar to the proof of (ii).
For the case $2$ of (iii), it is known that
\begin{center}
$\left<r^{2}\right>=\mathrm{ker} \rho_{1,3}^{n} \bigcap \left<r\right>=\mathrm{ker} \rho_{1,4}^{n} \bigcap \left<r\right> \subseteq \mathrm{ker} \rho_{1,2}^{n}\subseteq \mathrm{ker} \rho_{1,1}^{n}$.
\end{center}
When $\rho$ has exact two dimensional irreducible representation $\rho_{2,k_{i}}^{n},i=1,...,t$ in its decomposition, we see that
\begin{center}
$\mathrm{ker} \rho=\left<r^{2}\right>\bigcap (\bigcap \mathrm{ker} \rho_{2,k_{i}}^{n})=\left<r^{2}\right>\bigcap \left(\bigcap\left<r^{\frac{n}{(n,k_{i})}}\right>\right)$.
\end{center}
Therefore, similarly to the argument in (ii), it follows that $\rho$ is irreducible if and only if
$(\frac{n}{2},(n,k_{1}),...,(n,k_{t}))=1$, namely $(\frac{n}{2},k_{1},...,k_{t})=1$.
\end{proof}

\section{Main theorem}\label{sect:MT}

\hspace{1.3em}Compared with \cite[Theorem 1.1]{CST2021}, for the general Coxeter groups without infinite bonds in their Coxeter diagrams, we will prove the new version of the theorem through a faithful representation.

\begin{thm}\label{main thm}
Let $W$ be a Coxeter group with generators $\{s_{1},...,s_{n}\}$ associated to its coxeter digram without infinite bonds,
and $\rho$ be a faithful representation of $W$. If the proper joint spectrum $U$ relative to $\{s_{1},...,s_{n}\}$ of $\rho$ is known, then the Coxeter group can be determined by the set $U$.
\end{thm}

\begin{proof}
Take 2 generators $s_{i},s_{j}$ of $W$. The Theorem is equivalent to prove that the $m_{ij}=\mathrm{ord}(s_{i}s_{j})$ is determined by the set $U$.

Since $\rho$ is a faithful representation of $W$, we have $\rho$ is also a faithful representation of the dihedral group generated by $s_{i}$ and $s_{j}$.

Now, let $V_{ij}=\{(x_{1},...,x_{n})\in \C^{n}|x_{k}=0$ if $k\neq i$ or $j$\} and $U_{ij}=U\bigcap V_{ij}$.

By Theorem \ref{faithful}, we divide our argument into $3$ cases.

Case 1: When the set $U_{ij}$ consists of lines only, then by Theorem 2, we must have $m_{ij}=2$.

Case 2: When the $U_{ij}$ consists of some ellipses $E_{1},...,E_{t}$ and lines in $\{x_{i}+x_{j}-1=0, -x_{i}-x_{j}-1=0\}$.
Suppose $E_{j}$ is defined by the equation $x_{1}^{2}+x_{2}^{2}+2cos\frac{2\pi n_{j}}{m_{j}}x_{1}x_{2}-1=0$ for $j=1,...,t$
with $0<\frac{n_{j}}{m_{j}}<\frac{1}{2}, (n_{j},m_{j})=1$. By Table $2$ in Lemma \ref{tables},
we see $E_{j}$ is corresponding to an irreducible representation of $\left<s_{i},s_{j}\right>$ with kernel $\left<(s_{i}s_{j})^{m_{j}}\right>$.
By (ii) of Theorem \ref{faithful} and  the first case  of (iii) of Theorem \ref{faithful}, when $\rho$ is faithful, we have
 \begin{center}
 $(\frac{m_{ij}}{m_{1}},...,\frac{m_{ij}}{m_{t}})=1$,
\end{center}
and then $m_{ij}$ is the least common multiple of $m_{1},...,m_{t}$.

Case 3: When $U_{ij}$ consists of some ellipses $E_{1},...,E_{t}$ and lines in $\{x_{i}-x_{j}-1=0, x_{j}-x_{i}-1=0\}$. Therefore the order $m_{ij}$ must be even.

We suppose $E_{j}$ is in the form of case 2 in the above, by the second case  of (iii) of Theorem \ref{faithful}, it follows that
$$(\frac{m_{ij}}{2},\frac{m_{ij}}{m_{1}},...,\frac{m_{ij}}{m_{t}})=1.$$
Therefore, we suppose $\theta$ is the least common multiple of $m_{1},...,m_{t}$. When $2\nmid \theta$,  then we have
$m_{ij}=2\theta$; when $2|\theta$, it implies that  $m_{ij}=\theta$.
\end{proof}

	\begin{rem}
	From this paper, we observe that for Coxeter groups, the generating relations of the group, its representations, and their characteristic polynomials mutually determine one another. This establishes a trinity of unification among Coxeter groups, their faithful representations, and their geometric realizations. A natural question arises: what is the relationship between the representations of the Hecke algebra associated with a Coxeter group and the characteristic polynomials of these representations? Is there also a one-to-one correspondence? This remains an unresolved issue in our research.
	Moreover, while characteristic polynomials geometrically encapsulate groups or algebras along with their representations, what profound connections exist between characteristic polynomials and Kazhdan-Lusztig polynomials? Could this provide a powerful tool for studying the cells in Coxeter groups and exploring relations among Coxeter groups, intersection cohomology theory, and K-theory? These questions merit deep and sustained investigation.
\end{rem}

Shoumin Liu\\
Email: s.liu@sdu.edu.cn\\
School of Mathematics, Shandong University\\
Shanda Nanlu 27, Jinan, \\
Shandong Province, China\\
Postcode: 250100\\
Zhaohuan Peng\\
Email: 1391994462@qq.com\\
School of Mathematics, Shandong University\\
Shanda Nanlu 27, Jinan, \\
Shandong Province, China\\
Postcode: 250100\\
Xumin Wang\\
Email: 202320303@mail.sdu.edu.cn\\
School of Mathematics, Shandong University\\
Shanda Nanlu 27, Jinan, \\
Shandong Province, China\\
Postcode: 250100

\end{document}